\documentclass[11pt]{article}
\usepackage{latexsym,amsmath,stackrel,color,amsthm,epsfig,array}
\usepackage{graphicx}
\usepackage{amssymb}
\usepackage[left=0.9in,top=0.9in,right=0.9in,bottom=0.9in]{geometry}
\usepackage[linktocpage=true]{hyperref}
\usepackage{setspace}
\usepackage{mathrsfs}
\usepackage{caption}
\usepackage{titlesec}
\usepackage{float}
\usepackage{comment,caption}
\usepackage{tikz}
\usepackage{thmtools}
\usepackage[compress]{cite}

\newcounter{foo}
\makeatletter
\def\thm@space@setup{%
  \thm@preskip=\parskip \thm@postskip=0pt
}
\makeatother

\makeatletter
\renewenvironment{proof}[1][\proofname]{\par
  \vspace{-\topsep}
  \pushQED{\qed}%
  \normalfont
  \topsep8pt \partopsep0pt 
  \trivlist
  \item[\hskip\labelsep
        \itshape
    #1\@addpunct{.}]\ignorespaces
}{%
  \popQED\endtrivlist\@endpefalse
  \addvspace{6pt plus 6pt} 
}
\makeatother

\makeatletter
\def\thm@space@setup{%
  \thm@preskip=0.3cm
  \thm@postskip=0cm
}
\makeatother

\declaretheoremstyle[%
  spaceabove=6pt,%
  spacebelow=6pt,%
  headfont=\normalfont\itshape,%
  postheadspace=1em,%
  qed=\qedsymbol%
]{mystyle}

\def\qed{\hfill\ifhmode\unskip\nobreak\fi\quad\ifmmode\Box\else\hfill$\Box$\fi}
\def\ite#1{\hfill\break${}$\hbox to 50pt {\quad(#1)\hfill}}
\newtheorem{thm}{Theorem}[section]
\newtheorem{cor}[thm]{Corollary}

\newtheorem{lemma}[thm]{Lemma}
\newtheorem{conj}[foo]{Conjecture}
\newtheorem{prop}[thm]{Proposition}
\newtheorem{claim}[thm]{Claim}

\newtheorem{question}{Question}

\def\c{{\circlearrowright}}

\tikzstyle{vertex}=[circle,fill=black,inner sep=2pt]
\tikzstyle{vertrect}=[draw,rectangle,inner sep=2pt]
\tikzstyle{vertdia}=[draw,diamond,inner sep=2pt]

\newcommand{\h}{H}

\newcommand{\ep}{\epsilon}
\newcommand{\lam}{\lambda}

\newcommand{\Om}{\Omega}

\renewcommand{\l}{\left}
\renewcommand{\r}{\right}

\newcommand{\half}{\frac{1}{2}}

\newcommand{\sm}{\setminus}
\newcommand{\sub}{\subseteq}
\newcommand{\semi}{\widetilde{\c{C}}}
\newcommand{\valid}{\c{C}^r}
\newcommand{\vaLID}{\c{C}^3}

\renewcommand{\c}[1]{\mathcal{#1}}

\newcommand{\tr}[1]{\textrm{#1}}
\newcommand{\rec}[1]{\frac{1}{#1}}
\newcommand{\f}[2]{\frac{#1}{#2}}
\newcommand{\floor}[1]{\l\lfloor #1\r\rfloor}
\newcommand{\ceil}[1]{\l\lceil #1\r\rceil}

\newcommand{\mr}[1]{\mathrm{#1}}

\newcommand{\JO}[1]{\textcolor{blue}{#1}}

\newcommand{\TD}{\textbf{\textcolor{red}{TODO}}}

\renewcommand{\SS}[1]{\textcolor{red}{#1}}

\newcommand{\sat}{\mr{sat}}

\newcommand{\cex}{\text{\rm ex}_\circlearrowright}
\newcommand{\csat}{\text{\rm sat}_\circlearrowright}

\titleformat{\subsection}[runin]
{\normalfont\bfseries}{\thesubsection}{1em}{}

\begin{document}
\title{Saturation Problems in Convex Geometric Hypergraphs} \author{Jason O'Neill \footnote{Department of Mathematics, University of California, San Diego, 9500 Gilman Drive, La Jolla, CA 92093-0112, USA. E-mail:jmoneill@ucsd.edu. Research supported by NSF award DMS-1800332. } \\
\and 
Sam Spiro\thanks{Department of Mathematics, University of California, San Diego, 9500 Gilman Drive, La Jolla, CA 92093-0112, USA. E-mail: sspiro@ucsd.edu. This material is based upon work supported by the National Science Foundation Graduate Research Fellowship under Grant No. DGE-1650112.} }
\maketitle
\begin{abstract}
A convex geometric hypergraph (abbreviated cgh) consists of a collection of subsets of a strictly convex set of points in the plane. Extremal problems for cgh's have been extensively studied in the literature, and in this paper we consider their corresponding saturation problems. We asymptotically determine the saturation number of two geometrically disjoint $r$-tuples. Further, amongst the eight nonisomorphic $3$-uniform cgh's on two edges, we determine the saturation number for seven of these up to order of magnitude and the eighth up to a log factor. 

\end{abstract}

\section{Introduction}

A \textit{convex geometric hypergraph} (abbreviated cgh) $H$ is a collection of subsets (called edges) of a set of vertices $\Om_n=\{v_0,\ldots,v_{n-1}\}$ such that the vertices have a cyclic ordering $v_0<v_1<\cdots<v_{n-1}<v_0$.  It is convenient to view $\Om_n$ as the vertices of a circle in the plane with cyclic ordering based on the order they appear as one travels clockwise along the circle.  Two cgh's $H,H'$ are considered to be isomorphic if there exists a bijection between their vertex sets which respects their cyclic orderings and which induces a bijection between their edges.  We say that $H$ is $r$-uniform or an $r$-cgh if every $h\in H$ is of size $r$.
Extremal problems on convex geometric hypergraphs have been extensively studied \cite{ADMOS,FJKMV1,FKMV1,FMOV,KeP,KuP,HP,P2, FHK} in the literature. In this paper, we explore saturation problems in the convex geometric setting.  

The study of saturation problems was initiated by Erd{\H o}s, Hajnal, and Moon \cite{EHM} who determined the saturation number of graph cliques. Bollob{\'a}s \cite{BBOL} extended this result to complete $k$-graphs. Pikhurko \cite{OLEG} further studied upper bounds on the saturation number of finite classes of hypergraphs as well as the corresponding problems in the ordered hypergraph setting. For more on saturation problems and their history, see the survey \cite{FFS} and the references there within.

Given a cgh $F$, we say that $\h$ is {\em  $F$-free} if $\h$ does not contain a subhypergraph isomorphic to $F$. Let the extremal function $\cex(n,F)$ denote the maximum number of edges in an $F$-free $r$-cgh on $n$ points. We say that $\h$ is {\em  $F$-saturated} if $\h$ is $F$-free and if for all $e \in \binom{\Omega_n}{r} \setminus \h$, adding the edge $e$ to $\h$ creates a copy of $F$. Let the saturation function $\csat(n,F)$ denote the minimum number of edges in an $F$-saturated $r$-cgh on $n$ points. Let $M_1^{(r)}$ be the $r$-cgh consisting of two geometrically disjoint edges.  That is, $M_1^{(r)}$ has edges  $\{v_0,\ldots, v_{r-1} \}$ and $\{v_r, \ldots, v_{2r-1}\}$ with the ordering $v_0<\cdots<v_{2r-1}<v_0$; see Figure~\ref{fig:eightconfig} for a picture of $M_1^{(3)}$. 

In the classical extremal set theory setting, $M_1^{(r)}$ is a matching of size two, for which the extremal function is determined by the Erd{\H o}s-Ko-Rado theorem \cite{EKR} and the saturation number is bounded above in a few cases by F{\" u}redi \cite{FUREDI}. For $n$ large enough, these results give 
\[ \text{ex}(n, M_1^{(r)}) = \binom{n-1}{r-1} \hspace{3mm} \text{and}  \hspace{3mm}
\text{sat}(n, M_1^{(r)}) \leq \frac{3r^2}{4} \]  
where the upper-bound on the saturation number holds when there exists a projective plane of order $r/2$. In the convex geometric setting, determining $\cex(n,M_1^{(r)})$ was asked in \cite{FHK} and determined exactly in \cite{FMOV} when $r=3$. 
In this paper we show that asymptotically, $\csat(n,M_1^{(r)})$ is achieved by an $r$-cgh $H$ which consist of a star (i.e. every $r$-set containing a given vertex) together with a few extra edges.  This is perhaps surprising given that, in the classical setting, the star is the largest $M_1^{(r)}$-saturated $r$-uniform hypergraph. 

To be more precise, for $v_i,v_j\in \Om_n$ we define the interval $(v_i,v_j)=\{v_k: v_i<v_k<v_j<v_i\}$.  Similarly define, for example, $[v_i,v_j]=(v_i,v_j)\cup \{v_i,v_j\}$.   Let
\begin{equation}
H_n^{(r)}=\{ h \in \binom{\Omega_n}{r} : v_0 \in h\} \cup \bigcup\limits_{j=1}^{r-1} \{ h \in \binom{\Omega_n}{r} : v_{j} \in h,\ h\cap [v_{-r+j}, v_{n-1} ] \neq \emptyset \},\label{eq:const}\end{equation}
For example, $H_n^{(2)}$ consists of the star on $v_0$ together with the edge $\{v_1,v_{n-1}\}$, and $H_n^{(3)}$ consists of the star on $v_0$ together with every triple containing $\{v_1,v_{n-2}\}$,$\{v_1,v_{n-1}\}$, or $\{v_2,v_{n-1}\}$.  We will show that $H_n^{(r)}$ is $M_1^{(r)}$-saturated with $|H_n^{(r)}|\sim {n \choose r-1}$.  This will give the upper bound to the following theorem.

\begin{thm}\label{thm:mainsat}
For all $r \geq 2$,
\[\csat(n,M_1^{(r)}) \sim  {n\choose r-1}.\]
Moreover, every maximal $M_1^{(2)}$-saturated $2$-cgh has $n$ edges, and for $n\ge 6$ we have
\[\csat(n,M_1^{(3)}) ={n-1\choose 2} + 3n -11,\]
and $H_n^{(3)}$ the unique $3$-cgh up to isomorphism achieving this bound when $n>6$.
\end{thm}

We note that the bound $n>6$ is best possible, as in general for $n\le 2r$ every $M_1^{(r)}$-saturated cgh has the same number of edges.  A key tool in proving Theorem~\ref{thm:mainsat} is a structural result characterizing all $M_1^{(r)}$-saturated cgh's. See Theorem~\ref{thm:structure} for a precise statement of this result.
 
In this paper, we also consider $\csat(n,F)$ for all other two edge $r$-cgh when $r \leq 3$. There exists three nonisomorphic $2$-cghs with two edges: $G_0$ consisting of a vertex of degree two, $G_1$ consisting of two disjoint edges which cross, and $G_2$ consisting of two disjoint edges which do not cross. It is not difficult to see that $\csat(n,G_0) = \lfloor n/2 \rfloor$, and $\csat(n,G_1) = 2n-3$ since any $G_1$-saturated $2$-cgh must be a maximal outerplanar graph. By Theorem \ref{thm:mainsat}, it follows that $\csat(n,G_2) = n$. 

There are eight nonisomorphic $3$-cghs with two edges, and these are depicted in Figure~\ref{fig:eightconfig}.  F{\" u}redi, Mubayi, Verstraete and the first author \cite{FMOV} recently determined the extremal numbers of seven of these eight configurations asymptotically.

\begin{figure}[h]
    \centering
    \includegraphics[scale=.35]{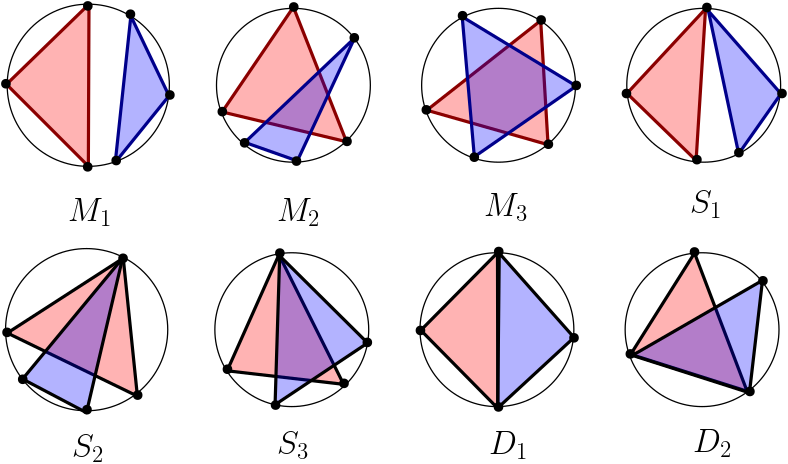}
    \caption{The eight $3$-cgh's with two edges.}
    \label{fig:eightconfig}
\end{figure}

We are able to determine the order of magnitude for the saturation number of seven of these configurations, and the saturation number of the eighth configuration up to a log factor: 
\begin{thm}\label{thm:order}
\begin{equation*}
\csat(n,F) = 
\begin{cases}
\Theta(n) \hspace{3mm} \text{if} \hspace{3mm}  F \in \{M_2,S_1,S_2\} \\
\Theta(n^2) \hspace{3mm} \text{if} \hspace{3mm}  F \in \{M_1,M_3,D_1,D_2\}. \\
\end{cases}   
\end{equation*}
Further,
\[\Om(n)=\csat(n,S_3)=O(n\log_2 n)\]
\end{thm}

\subsection*{Organization and Notation} We prove Theorem \ref{thm:order} in Section \ref{sec:proofoforder} and prove Theorem \ref{thm:mainsat} in Section \ref{sec:proofofm1}. In the Appendix we use more involved arguments to get sharper asymptotic bounds on $\csat(n,F)$ for the $3$-cgh's $F$ considered in Theorem~\ref{thm:order}.

We let $\Om_n$ denote a set of cyclically ordered points $\{v_0<\ldots<v_{n-1}<v_0\}$. We use capital letters to denote cghs and lower case letters to denote their edges.  As much as possible we use the convention that edges of $H$ are denoted by $h$, and that $e$ denotes an arbitrary set which may or may not be in $H$. Given a cgh $H \subset \binom{\Om_n}{r}$ and a vertex $v_i \in \Om_n$, let $d_H(v_i):=|\{ h \in H: v_i \in h \}|$ denote the degree of $v_i$ in $H$. For functions $f,g :\mathbb N\rightarrow \mathbb R^+$, we write $f = o(g)$ if $\lim_{n \rightarrow \infty} f(n)/g(n) = 0$, and $f = O(g)$ if there is $c > 0$ such that $f(n)\leq cg(n)$ for all $n \in \mathbb N$. If $f = O(g)$ and $g = O(f)$, we write $f =\Theta(g)$.

\section{Proof of Theorem~\ref{thm:order}}\label{sec:proofoforder}
In this section we prove Theorem~\ref{thm:order}, establishing the order of magnitude of $\csat(n,F)$ for every $F\ne S_3$ which is a $3$-uniform cgh on two edges.  Throughout the proof we do not concern ourselves with trying to optimize our asymptotic bounds, and we refer the reader to the Appendix for a discussion on how to obtain sharper estimates. We first show the following general result.

\begin{prop}\label{prop:oleg}
Let $\mathcal{F}$ be a finite collection of $r$-uniform cgh's.  Then $\csat(n,\c{F})=O(n^{r-1})$.
\end{prop}
For this proof we utilize ordered hypergraphs, which are defined exactly the same as cgh's except with a linear ordering of their vertex set.
\begin{proof}
For $H\sub {\Om_n\choose r}$ an $r$-cgh, let $\overrightarrow{H}$ be the $r$-uniform ordered hypergraph obtained from $H$ by giving $\Om_n$ the linear ordering $v_0<v_1<\cdots<v_{n-1}$.
\begin{claim}
There exists a finite collection $\mathcal{F}'$ of ordered $r$-uniform hypergraphs such that $H\sub {\Om_n\choose r}$ is $\c{F}$-saturated as an $r$-cgh if and only if $\overrightarrow{H}$ is $\c{F}'$-saturated as an $r$-uniform ordered hypergraph.
\end{claim}
\begin{proof}
For simplicity we prove the result only when $\c{F}=\{F\}$. If $F$ is a cgh on $U=\{u_1<u_2<\cdots<u_k<u_1\}$, let $F_i$ be the ordered $r$-uniform hypergraph obtained from $F$ by giving $U$ the linear ordering $u_i<u_{i+1}<\cdots<u_k<u_1<\cdots<u_{i-1}$.  It then follows that $H$ is $F$-free if and only if $\overrightarrow{H}$ is $\{F_1,\ldots,F_k\}$-free.  In particular, $H$ is $F$-saturated if and only if $\overrightarrow{H}$ is $\{F_1,\ldots,F_k\}$-saturated.
\end{proof}
In \cite{OLEG}, Pikhurko proved that every finite collection of $r$-uniform ordered hypergraphs $\c{F}'$ has saturation number $O(n^{r-1})$.  This combined with the claim above gives the result.
\end{proof}

With this we can now prove Theorem~\ref{thm:order}.

\begin{proof}[Proof of Theorem~\ref{thm:order} assuming Theorem~\ref{thm:mainsat}]
The result for $M_1$ is dealt with in Theorem~\ref{thm:mainsat}, so it suffices to consider the remaining cases.

\underline{Lower bounds.} Observe that if $H$ is $S_1$ or $S_2$-free and has three isolated vertices $u,v,w$, then $H+\{u,v,w\}$ is still $S_1$ or $S_2$-free.  Thus every $S_1$ or $S_2$-saturated cgh has at most 2 isolated vertices, which proves $\csat(n,S_i)\ge (n-2)/3$ for $i=1,2$.

Note that any $M_2$-saturated or $S_3$-saturated cgh $H$ has $\{v_i,v_{i+1},v_{i+2}\}\in H$, since such an edge can never be used to form an $M_2$ or $S_3$.  Thus $\csat(n,M_2),\csat(n,S_3)\ge n$.  Similarly every $M_3$-saturated cgh contains every edge of the form $\{v_i,v_{i+1},v_j\}$ with $j\ne i,i+1$, which proves $\csat(n,M_3)\ge n(n-3)$.

To deal with $D_i$, let $G_i$ be the graph with $V(G_i) = \binom{\Omega_n}{3}$ such that $e,f \in \binom{\Omega_n}{3}$ form an edge if $\{e,f\}$ are isomorphic to $D_i$ as cgh's.
Given a $D_i$-saturated cgh $H \subset \binom{\Omega_n}{3}$, it follows that $N_{G_i}(H) = \binom{\Omega_n}{3} \setminus H$. It is straightforward to check that $d_{G_1}(e) = n-3$ for all $e \in \binom{\Omega_n}{3}$ (i.e., every $e\in {\Om_n\choose 3}$ forms a $D_1$ with exactly $n-3$ other $f\in {\Om_n\choose 3}$), so if $H$ is $D_1$-saturated, 
\[ |H| \cdot (n-3) \ge N_{G_1}(H)= \binom{n}{3} - |H|,    \]
which implies \[\csat(n,D_1) \geq \f{n(n-1)}{6}.\]  Similarly we have $d_{G_2}(e)=2(n-3)$ for all $e\in {\Om_n\choose 3}$, which gives \[\csat(n,D_2)\ge \f{n(n-1)(n-2)}{6(2n-5)}\sim \f{n^2}{12}.\]

\underline{Upper bounds.}  The quadratic upper bounds for $M_3,D_1,D_2$ follow from Proposition~\ref{prop:oleg}.  For $S_2$, take $H$ to be the cgh consisting of all edges of the form $\{v_i,v_{i+1},v_{i+2}\}$. For $n\ge 4$, $H$ is $S_2$-saturated which show $\csat(n,S_2)\le n$.  The remaining constructions will be slightly more complicated.

\begin{claim}
$\csat(n,S_2)\le n+O(1)$.
\end{claim}

\begin{proof}
Let $n=4q+k$ with $q\ge 1$ and $0\le k<4$. We will define an $S_2$-saturated cgh $H \subset \binom{\Omega_n}{3}$ which has $q$ components $H_0, \ldots, H_{q-1}$. For $1 \leq \ell \leq q-1$, let $H_\ell$ consist of the complete $3$-cgh on $\{v_{4\ell+k}, v_{4\ell+k+1}, v_{4\ell+k+2}, v_{4\ell+k+3} \} $. Let $H_0'$ be the cgh on $\Omega_{4+k}$ consisting of the complete $3$-cgh on $\Omega_4=\{ v_0, v_1, v_2, v_3\}$ and all edges of the form $\{v_1,v_2,v_i\}$ for $3 \leq i \leq 3+k$, and let $H_0 \supseteq H_0'$ be \textit{any} $S_2$-saturated cgh on $\Omega_{4+k}$ which contains the cgh $H_0$. It is straightforward to check that $H= H_0 \cup \cdots \cup H_{q-1} \subset \binom{\Omega_n}{3}$ is $S_2$-free and $|H|=n+O(1)$. 

Seeking to prove that $H$ is $S_2$-saturated, it suffices to check $3$-sets which meet at least two components since each component of $H$ is $S_2$-saturated. Let $e = \{v_i,v_j,v_k\} \notin H$, and without loss of generality assume $v_i$ is in a component $H_\ell$ with $v_j,v_k\notin H_\ell$. It is not hard to check that there exists an edge $h\in H_\ell$ such that $v_i \in h$ with $i$ either the smallest or largest index of a vertex in $h$. Without loss of generality we can assume $h = \{ v_i, v_x, v_y \} \in H_\ell$ with $v_i<v_x<v_y<v_i$. This means $\{e,h\}$ forms an $S_2$ since $v_i < v_x < v_y < v_j, v_k < v_i$ and hence $H$ is $S_2$-saturated as desired.
\end{proof}

\begin{claim}
$\csat(n,M_2)\le 3n-2$.
\end{claim}
\begin{proof}
For $n\ge 6$, define $H = C \cup H_1 \cup H_2 \cup H_3$ where 
\begin{eqnarray*}
 C &=& \{ \{v_i,v_{i+1}, v_{i+2}\} : 0 \leq i \leq n-1 \} ; \hspace{3mm} H_1 = \{ e \in \binom{\Omega_n}{3}: \{v_1,v_4\} \subset e \} \\ 
H_2 &=& \{ e \in \binom{\Omega_n}{3}: \{v_0,v_1\} \subset e \} ; \hspace{3mm}
H_3 = \{ \{v_1,v_3,v_5\}, \{v_0,v_2,v_4\}, \{v_0,v_4,v_5\}, \{v_1,v_2,v_5\} \}. 
\end{eqnarray*}
It is straightforward to see that $|H| = n+(n-2)+(n-4)+4 = 3n-2$.  One can check that every disjoint pair of edges of $H$ lies in $C\cup H_3$ and that these pairs only form $M_1$'s and $M_3$'s.  Thus $H$ is $M_2$-free. To see that $H$ is $M_2$-saturated, let $e \in \binom{\Omega_n}{3} \setminus H$ and consider $h=\{v_0,v_1,v_4\}$. 

\textit{Case 1:} We have $|e \cap h|=2$.  Since $e\notin H_1\cup H_2\sub H$, we necessarily have $\{v_0,v_4\} \subset e$, and hence $e$ forms an $M_2$ with $\{v_1,v_2,v_5\} \in H$.

\textit{Case 2:} We have $|e\cap h|=1$, say with $v_0\in e$ (which means $v_1,v_4\notin e$).  If $e \cap \{v_2,v_3\} = \emptyset$, then it is not hard to see that $e$ forms an $M_2$ with $\{v_1,v_4,v_x\}$ for an appropriately chosen $v_x$ since $e\notin C$. 
We next deal with the subcase $e\cap \{v_2,v_3\}\ne \emptyset$.  If $e=\{v_0,v_2,v_3\}$, then this forms an $M_2$ with $\{v_1,v_4,v_5\}$. If $e=\{v_0,v_2, v_x\}$ with $x\ne 3$, then this forms an $M_2$ with $\{v_1,v_3,v_4\}$, and if $e=\{v_0,v_3, v_x\}$ with $x\ne 2$ this forms an $M_2$ with $\{v_1,v_2,v_4\}$.  This deals with all the possible cases with $v_0\in A$.  The cases for $v_1 \in A$ and $v_4 \in e$ can be similarly worked out. 

\textit{Case 3:} We have $e \cap h = \emptyset$. Then $e$ forms an $M_2$ with $\{ v_0,v_1,v_4\}$ unless $A \cap 
\{v_0,\ldots,v_5\} = \emptyset$. In the latter case, $e$ forms an $M_2$ with $\{v_1,v_4,v_x\}$ for an appropriately chosen $v_x$.
\end{proof}
\begin{claim}
$\csat(n,S_3)\le 3 n\log_2 n$.
\end{claim}
\begin{proof}
We define an $n$-vertex $S_3$-saturated cgh $H_n$ inductively as follows. For $n=0,1,2$ the cgh $H_n$ is empty.  For larger $n$, we start by defining $H_n'$ to consist of a copy of $H_{\floor{n/2}-1}$ placed in $(v_0,v_{\floor{n/2}})$, a copy of $H_{\ceil{n/2}-1}$ placed in $(v_{\floor{n/2}},v_0)$, and every edge containing $\{v_0,v_{\floor{n/2}}\}$.  It is not difficult to see that $H'_n$ is $S_3$-free, and we let $H_n$ consist of any $S_3$-saturated cgh containing $H'_n$.

Some casework\footnote{If $e\sub (v_0,v_{\floor{n/2}})$ then this would give a contradiction since $H_{\floor{n/2}-1}$ is $S_3$-saturated.  A similar argument holds for $e\sub (v_{\floor{n/2}},v_0)$,  and every other case follows by considering an appropriate edge containing $\{v_1,v_{\floor{n/2}}\}$.} shows if $e\in H_n\sm H_n'$, then either $e=\{v_{n-1},v_0,v_1\},\ e=\{v_{\floor{n/2}-1},v_{\floor{n/2}},v_{\floor{n/2}+1}\}$, or $e=\{v_i,v_{j},v_{j+1}\}$ with $i=0,\floor{n/2}$ and $j\ne n-1,0,\floor{n/2}-1,\floor{n/2}$.   Thus if $f(n):=|H_n|$,  
\[f(n)\le 3n+f(\floor{n/2}-1)+f(\ceil{n/2}-1).\]
We now prove by induction that $f(n)\le 3n\log_2(n)$ for all $n$, the base cases for  $n\le 2$ being trivial.  By the recursive formula above, 
\[f(n)\le 3n+3\floor{n/2} \log_2(\floor{n/2})+3(\ceil{n/2}-1)\log_2(\ceil{n/2}-1)\le 3n+3n \log_2(n/2)=3n\log_2 n.\]
We conclude that $|H_n|=f(n)\le 3n\log_2(n)$, proving the result.
\end{proof}
This establishes the desired order of magnitude for each $F$, proving the result.
\end{proof}

\section{Proof of Theorem \ref{thm:mainsat}}\label{sec:proofofm1}
If $C=(w_1,\ldots,w_{2\ell+1})$ is a $(2\ell+1)$-tuple of distinct vertices of $\Om_n$ with $\ell\ge 1$ such that \[w_1<w_3<w_5<\cdots<w_{2\ell+1}<w_2<w_4<\cdots<w_{2\ell}<w_1 ,\] then we define 
\[ H_n^{(r)}(C) := \{ e \in \binom{\Omega_n}{r}: e \cap [w_i,w_{i-1}] \neq \emptyset \; \forall i \in [2\ell+1]    \},\]
where here and throughout we write the indices of the $w_i$ modulo $2\ell+1$.  As a warmup, we will show all $M_1^{(2)}$-saturated $2$-cgh's are an odd cycles together with a set of leaves.

\begin{figure}[h]
\centering
\includegraphics[scale=.25]{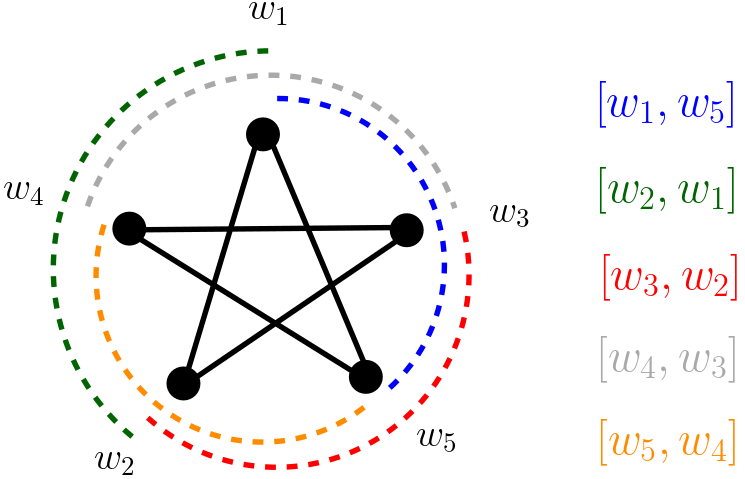}
\caption{For the five tuple $C=(w_1,w_2,w_3,w_4,w_5)$, the edges of $H_n^{(r)}(C)$ intersects each of the five dashed intervals.  In particular, $H_n^{(r)}(C)$ contains every $r$-set which contains one of the black edges.}
\label{fig:satm1}
\end{figure}

\begin{lemma}\label{lem:r2}
Let $C=(w_1,\ldots,w_{2\ell+1})$ with $w_1<w_3<\ldots<w_{2\ell+1}<w_2<\ldots<w_{2\ell}<w_1 \in \Omega_n$.  The cgh $H_n^{(2)}(C)$ consists of all of the edges $\{w_i,w_{i+1}\}$ for $1\le i\le 2\ell+1$, together with every edge $\{v_j,w_i\}$ for $v_j\in (w_{i-1},w_{i+1})$ and in particular, every $H_n^{(2)}(C)$ has exactly $n$ edges.
\end{lemma}

\begin{proof}
We first claim that the neighborhood of every $w_i$ vertex is $[w_{i-1},w_{i+1}]$, and by the symmetry of the definitions it suffices to prove this when $i=1$.  Observe that $w_1$ is not contained in the intervals $[w_3,w_2]$ nor $[w_{2\ell+1},w_{2\ell}]$, so every neighbor of $w_1$ must lie in $[w_3, w_2]\cap [w_{2\ell+1},w_{2\ell}]=[w_{2\ell+1},w_2]$.  We claim that any pair $\{w_1,v_j\}$ with $v_j\in [w_{2\ell+1},w_2]$ intersects any interval $[w_i,w_{i-1}]$.  Indeed, $w_1\in [w_i,w_{i-1}]$ if $i$ is even and $v_j\in [w_i,w_{i-1}]$ if $i$ is odd.  This proves that the neighbors of $w_1$, and hence of every $w_i$, are as claimed.

Next consider $v_j\in (w_{2\ell+1},w_{2})$.  We claim that the unique neighbor of $v_j$ is $w_1$.  Indeed, any neighbor of $v_j$ must be in $[w_{2},w_1]\cap [w_1,w_{2\ell+1}]=\{w_1\}$, and the analysis above shows that this is indeed a neighbor.  Again the symmetry of the situation shows that every $v_j\in (w_{i-1},w_{i+1})$ is adjacent to $w_i$ and no other vertices.  This proves the result.
\end{proof}

Our main tool for this section is the following characterization of $M_1^{(r)}$-saturated cgh's.

\begin{thm}\label{thm:structure}
Let $r\ge2$ and $n\ge 2r$. An $r$-cgh $H \subset \binom{\Omega_n}{r}$ is  $M_1^{(r)}$-saturated if and only if $H = H_n^{(r)}(C)$ for some $C=(w_1,\ldots,w_{2\ell+1})$ with $\ell\ge 1$,  $w_1<w_3<\ldots<w_{2\ell+1}<w_2<\ldots<w_{2\ell}<w_1$, and $|[w_i,w_{i-1}]| \geq r$ for all $i$.
\end{thm}

The backwards direction of this statement is relatively easy to prove.

\begin{lemma}\label{lem:backwards}
Let $C=(w_1,\ldots,w_{2\ell+1})$ with $w_1<w_3<\ldots<w_{2\ell+1}<w_2<\ldots<w_{2\ell}<w_1 \in \Omega_n$.  If $|[w_i,w_{i-1}]|\ge r$ for all $i$, then $H_n^{(r)}(C)$ is $M_1^{(r)}$-saturated.
\end{lemma}

\begin{proof}
Let $H=H_n^{(r)}(C)$.  Assume for contradiction that there exist distinct $h_1,h_2\in H$ forming an $M_1^{(r)}$, say with $h_1\sub [v_1,v_j]$ and $h_2\sub (v_j,v_1)$.  Because $h_1,h_2\in H$, for all $i$ we must have $h_1,h_2\cap [w_i,w_{i-1}]\ne \emptyset$, and hence either $w_i\in [v_1,v_j]$ and $w_{i-1}\in (v_j,v_1)$, or $w_{i-1}\in [v_1,v_j]$ and $w_i\in (v_j,v_1)$.  If, say, $w_{2\ell+1}\in [v_1,v_j]$, then it is straightforward to prove by induction that $w_i\in [v_1,v_j]$ for all odd $i$, but this implies $w_{2\ell+1}=w_{1-1}\in (v_j,v_1)$, a contradiction.  Thus $H$ is $M_1^{(r)}$-free.

Consider any $e\in {\Om_n\choose r}\sm H$. By definition there must exist some $i$ such that $e\cap [w_i,w_{i-1}]=\emptyset$, and without loss of generality we can assume $i=1$.  Let $e'$ be any $r$-set containing $w_1,w_{2\ell+1}$ and with $e'\sub [w_1,w_{2\ell+1}]$, which exists by hypothesis.   It is not difficult to see that $e'\cap [w_{i'},w_{i'-1}]\ne \emptyset $ for any $i'$ since $w_1,w_{2\ell+1}\in e'$, so $e'\in H$.  Because $e\cap [w_1,w_{2\ell+1}]=\emptyset$, this set forms an $M_1^{(r)}$ with $e'\sub [w_1,w_{2\ell+1}]$. As $e$ was an arbitrary non-edge, we conclude that $H$ is $M_1^{(r)}$-saturated.
\end{proof}

In the upcoming subsection we prove the forward direction of Theorem~\ref{thm:structure}.  Given this theorem, determining $\csat(n,M_1^{(r)})$ is equivalent to determining $\min |H_n^{(r)}(C)|$ with $C$ as in Theorem~\ref{thm:structure}.  This reduces to a complicated optimization problem, which is (asymptotically) solved in Subsection~\ref{ss:opt}.
\subsection{Proof of Theorem~\ref{thm:structure}}

We introduce the notion of ``nearest leftmost" and ``nearest rightmost" neighbors of a vertex $v_i \in \Omega_n$ in an $M_1^{(r)}$-saturated cgh.   More precisely, given a vertex $v_i \in \Omega_n$ and an $M_1^{(r)}$-saturated $H$ with $d_H(v_i) >0$, define $\lam(v_i)$ to be the unique vertex $v_j$ such that there exists $h \in H$ with $v_i \in h$ and $h \subseteq [v_j,v_i]$ and such that there does not exists an edge $h\in H$ with $v_i\in h$ and $h\sub [v_{j+1},v_i]$.     Similarly define $\rho(v_i)$ to be the unique vertex $v_j$ such that there exists $h \in H$ with $v_i \in h$ and $h \subseteq [v_i,v_j]$ but there does not exist such an edge contained in $[v_i,v_{j-1}]$. 

\begin{figure}[h]
\centering
\includegraphics[scale=.25]{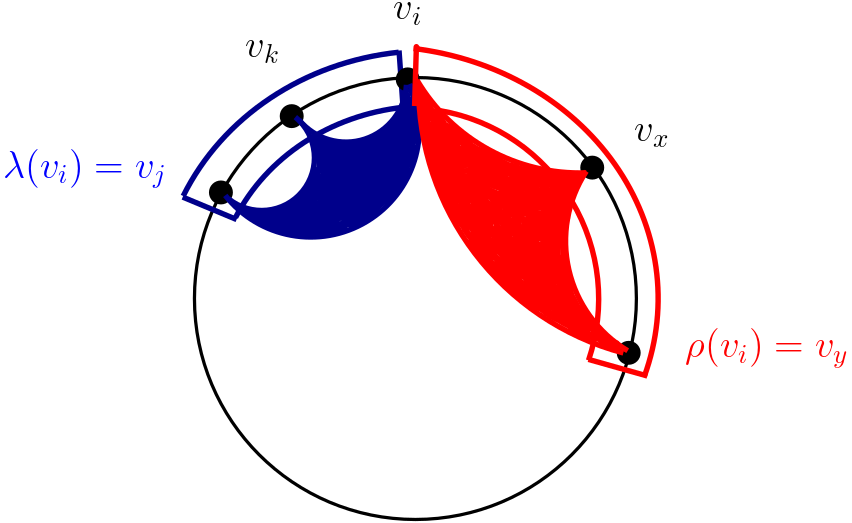}
\caption{A depiction of the case $r=3$ with $\lam(v_i)=v_j$ and $\rho(v_i)= v_y$.}
\label{fig:rholamdef}
\end{figure}

We first show some properties of $\lam(v_i)$ and $\rho(v_i)$ in $M_1^{(r)}$-saturated cgh's.

\begin{prop}\label{prop:basic}
Let $H \subset \binom{\Omega_n}{r}$ be $M_1^{(r)}$-saturated with $n\ge 2r$. Then the following hold: 
\begin{enumerate}
\item[(1)] Every $v_i\in \Omega_n$ has positive degree.  In particular, $\lam(v_i),\rho(v_i)$ are always well defined. \item[(2)] For all $v_i \in \Omega_n$, $\lambda(v_i)\notin [v_{i-r+2},v_{i+r-1}]$ and $\rho(v_i)\notin [v_{i-r+1},v_{i+r-2}]$. 
\item[(3)] For all $v_i \in \Omega_n$, $\lambda(v_i) < v_i< \rho(v_i) \le \lambda(v_i)$.
\item[(4)] If $ v_j \in [\rho(v_i), \lam(v_i)]$, then every $e \in \binom{\Omega_n}{r}$ with $\{ v_i,v_j\} \subseteq e$ is an edge in $H$. 
\end{enumerate}

\end{prop}
\begin{proof}
We first prove (1). Assume $d(v_i)=0$. Let $u_1 \in \Omega_n$ be the unique vertex of positive degree such that $[v_i,u_1)$ contains no vertex of positive degree.  Let $h=\{u_1<u_2<\cdots<u_r<u_1\} \in H$ be so that $u_r=\rho(u_1)$ and consider $e=h\sm \{u_1\}\cup \{v_i\}$. Note $e \notin H$ and hence there exist $h'\in H$ forming an $M_1^{(r)}$ with $e$, i.e. which lies in one of the intervals $(v_i,u_2),(u_2,u_3),\ldots,(u_r,v_i)$.  Because $h'$ does not form an $M_1^{(r)}$ with $h$, we can not have $h'\sub (u_j,u_{j+1})$ for any $j$ nor $h'\sub (u_r,u_1)\sub (u_r,v_i)$, and hence $h'\sub (v_i,u_2)$. Since every vertex in $(v_i,u_1)$ has degree $0$, and because $h'\not\sub (u_1,u_2)$, we must have $u_1\in h'$ and $h'\sub [u_1,u_2)$.  Because $u_1\in h'\sub [u_1,u_2)\subsetneq [u_1,u_r]$, this edge contradicts the assumption $\rho(u_1)=u_r$, proving the result. 

We next prove (2). We only consider the result for $\lam(v_i)$, as the $\rho(v_i)$ case is completely analogous. By definition, $\lam(v_i)\notin [v_{i-r+2},v_{i}]$.   Suppose for contradiction that $\lam(v_i)\in [v_{i+1},v_{i+r-1}]$. By definition of $\lam(v_i)$, there exists $h_\lam=\{u_1<\cdots<u_r<u_1\} \in H$ with $u_1=\lam(v_i),\ u_r=v_i$, and as $n \geq 2r$, a $w \in [\lam(v_i),v_i]$ with $w \notin h_\lam$.  Let $e_\lam= h_\lam \sm \{\lam(v_i)\} \cup \{w\}$, and note that by definition of $\lam(v_i)$ we have $e_\lam\notin H$. This means there exists some $h \in H$ so that $h$ and $e_\lam$ form a copy of $M_1^{(r)}$. If $h \sub (\lam(v_i),v_i)$, then $h\sub (u_j,u_{j+1})$ for some $j$ and hence $\{h,h_\lam\}$ form an $M_1^{(r)}$, a contradiction.  Thus $h\sub [v_{i+1},\lam(v_i)]$; a contradiction as $|[v_{i+1},\lam(v_i)]| \leq r-1$.

For (3), let $h_\lam=\{u_1<\cdots<u_r\} \in H$  be such that $u_1=\lam(v_i)$ and $u_r=v_i$, and similarly define an edge $h_\rho$. Assume for contradiction that $\rho(v_i) \notin [v_i, \lambda(v_i)]$.  This and (2) imply that $\lam(v_i)\ne v_{i-r+1}$, and in particular that $h_\lam\ne [v_{i-r+1}, v_i]$.  Thus there exists some element $w\in [\lam(v_i),v_i]$ such that $w\notin h_\lam$.

Let $e_\lam:=h_\lam \sm \{\lam(v_i)\}\cup \{w\}$, noting $e_\lam\notin H$ by the definition of $\lam(v_i)$. Thus there exist $h_1\in H$ which forms an $M_1^{(r)}$ with $e_\lam$.  Because $h_1$ does not form an $M_1^{(r)}$ with $h_\lam$, it is not difficult to see that $h_1\sub (v_i,\lam(v_i)]$. An analogous argument gives that there exists $h_2 \in H$ with $h_2\sub [\rho(v_i),v_i)$. By assumption $(v_i,\lam(v_i)]$ and $[\rho(v_i),v_i)$ are disjoint, so $h_1,h_2$ form an $M_1^{(r)}$, a contradiction.

We next prove (4). Assume that $v_j,e$ are as in the hypothesis with $e\notin H$.  Because $H$ is saturated, $e$ forms an $M_1^{(r)}$ with some edge $h \in H$, and thus $ h \sub (v_i,v_j)$ or $h \sub (v_j,v_i)$. By Proposition~\ref{prop:basic}(3) and our hypothesis, $\lam(v_i)<v_i<\rho(v_i)\le v_j$.  If $h\sub (v_i,v_j)$, then by definition of $\lam(v_i)$ there exists an edge $h'\sub [\lam(v_i),v_i]$, which forms an $M_1^{(r)}$ with $h$, a contradiction. We reach a  similar contradiction if  $h\sub (v_j,v_i)$, proving the result. 
\end{proof}

The next proposition relates $\lam(v_i),\rho(v_i)$ and $\lam(v_j),\rho(v_j)$ for different vertices $v_i, v_j \in \Omega_n$. 

\begin{prop}\label{prop:relativerholam}
Let $H \subset \binom{\Omega_n}{r}$ be $M_1^{(r)}$-saturated with $n\ge 2r$. Then the following hold:
\begin{itemize}
    \item[(1)] For all $v_i \in \Omega_n$, $\rho(v_{i+1})=\lam(v_i)$. 
    \item[(2)] For all distinct $v_i,v_j \in \Omega_n$, $|[\rho(v_i),\lam(v_i)]\cap [\rho(v_j),\lam(v_j)]|\le 1$.
    \item[(3)] For all $j$, exactly one interval of the form $[\rho(v_i),\lam(v_i)]$ contains both points $v_j,v_{j+1}$.
    \item[(4)] If $\lam(v_i)\ne \rho(v_i)$, then $\lam(\rho(v_i))=v_i$ and $\rho(\rho(v_i))\in [\lam(v_i),v_i)$.
\end{itemize}
    
\end{prop}

\begin{proof}
We start with (1). Since $\lam(v_i)\notin[v_{i-r+2},v_{i+r-2}]$ by Proposition~\ref{prop:basic}(2), it follows from Proposition~\ref{prop:basic}(4) that $h_1,h_2 \in H$ where 
$h_1=\{v_i,v_{i+1},\ldots,v_{i+r-2},\lam(v_i)\}$ and $h_2=\{v_i,v_{i-1},\ldots,v_{i-r+2},\lam(v_i)\}.$
As such, no edge in $H$ is contained entirely in $(\lam(v_i),v_i)$ or $(v_i,\lam(v_i))$, and hence every edge containing $v_{i+1}$ must contain an element in $[\lam(v_i),v_i]$, so $\rho(v_{i+1})\in [\lam(v_i),v_i]$.
	
By Proposition~\ref{prop:basic}(2),   $e=\{v_{i+1},v_{i+2},\ldots,v_{i+r-1},\lam(v_i)\}$ is a set of size $r$.  We claim that $e\in H$, which will imply $\rho(v_{i+1})\in [v_{i+1},\lam(v_i)]$ and hence $\rho(v_{i+1})=\lam(v_i)$.  Indeed, if $e \notin H$, then there exists $h \in H$ forming an $M_1^{(r)}$ with $e$, and in particular $h\sub (v_{i+1},\lam(v_i))$ or $h\sub (\lam(v_i),v_{i+1})$.  The former case is impossible as mentioned above, so $h\sub (\lam(v_i),v_{i+1})$.  Also $v_i\in h$ as otherwise $h,h_1$ would form an $M_1^{(r)}$.  But then $h\sub (\lam(v_i),v_i]$ contradicts the definition of $\lam(v_i)$, proving the result.

For (2), suppose $w_1,w_2 \in [\rho(v_i),\lam(v_i)]\cap [\rho(v_j),\lam(v_j)]$ with $w_1\ne w_2$. Using Proposition~\ref{prop:basic}(3), we can assume without loss of generality that $v_i<v_j<w_1<w_2<v_i$. Note that Proposition~\ref{prop:basic}(2) implies that $[v_j,w_1]\supseteq [v_j,\rho(v_j)]$ contains at least $r$ elements, with this also holding for $[w_2,v_i]$. Thus if $e_1$ is any $r$-set in $[v_j,w_1]$ containing $v_j,w_1$, then $e_1\in H$ by Proposition~\ref{prop:basic}(4). Similarly there exists $e_2\sub [w_2,v_i]$ in $H$, and hence $e_1,e_2$ form an $M_1^{(r)}$ in $H$, a contradiction.

We next consider (3). 	The pair is in at most one such interval by part (2) of this proposition. For all $v_i \in \Omega_n$, note that $[\rho(v_i),\lam(v_i)] = [\rho(v_i), \rho(v_{i+1}]$ by part (1) of this proposition. By Proposition~\ref{prop:basic}(3), $\rho(v_0)\le \rho(v_1)\le \rho(v_2)\le \cdots\le\rho(v_0)$, so  $v_j,v_{j+1}$ are in one of these intervals.

Finally, we consider (4). By Proposition~\ref{prop:basic}(3), $v_i<\rho(v_i)<\lam(v_i)<v_i$.  By definition, there exists $h_\rho\sub [v_i,\rho(v_i)] \in H$ with $v_i,\rho(v_i)\in h_\rho$ and also some $h_\lam \in H$ defined in an analogous way.  The existence of $h_\rho$ implies $\lam(\rho(v_i))\in [v_i,\rho(v_i)]$.  If there exists $h\sub (v_i,\rho(v_i)]$, then it forms an $M_1^{(r)}$ with $h_\lam$, a contradiction, so  $\lam(\rho(v_i))\sub (\rho(v_i),v_i]$, and thus $\lam(\rho(v_i))=v_i$.
	
For the second part, the existence of $h_\lam$ implies that any edge $h \in H$ containing $\rho(v_i)$ has a vertex in $[\lam(v_i),v_i]$, so $\rho(\rho(v_i))\in [\lam(v_i),\rho(v_i))$.  Take $e$ to be any $r$-set containing $\rho(v_i),v_{i-1}$ and which is disjoint from $[v_i,\rho(v_i))$, which exists by Proposition~\ref{prop:basic}(2). By Proposition~\ref{prop:basic}(4), $h:=e\sm \{v_{i-1}\}\cup\{ v_i\}\in H$. If $e\notin H$, then $h\in H$ implies that it must form an $M_1^{(r)}$ with some $h'\sub [v_i,\rho(v_i))$ which contains $v_i$.  This contradicts the definition of $\rho(v_i)$, so $e\in H$ and hence $\rho(\rho(v_i))\in [\lam(v_i),v_i)$.
\end{proof}

We next prove the existence of the $w_i$'s from Proposition~\ref{thm:structure}.

\begin{lemma}\label{lem:W}
Let $H \subset \binom{\Omega_n}{r}$ be $M_1^{(r)}$-saturated with $n\ge 2r$. Then there exist $\ell\ge 1$ and distinct vertices $w_1,\ldots,w_{2\ell+1} \in \Omega_n$ such that $w_1<w_3<w_5<\cdots<w_{2\ell+1}<w_2<w_4<\cdots<w_{2\ell}<w_1$ with $\lam(w_i)=w_{i+1}$ and $\rho(w_i)=w_{i-1}$ for all $i$ with the indices written mod $2\ell+1$.
\end{lemma}

\begin{proof}
By Proposition \ref{prop:relativerholam}(3), there exists a vertex $w_1$ with $\rho(w_1)\ne \lam(w_1)$. Continuing, we define $w_i=\lam(w_{i-1})$ for all $i$.  It is straightforward to prove using Proposition~\ref{prop:relativerholam}(1) and (4) that for all $i\ge 2$ we have $\rho(w_i)=w_{i-1}$ and $\rho(w_i)\ne\lam(w_i)$.  As $\Omega_n$ is finite, there exists a pair of integers $k'\le k$ such that $w_{k+1}=w_{k'}$, and without loss of generality we can assume $k'=1$.

\begin{claim}
Let  $2\le p\le k$. Then 
\begin{eqnarray*}
w_1 &<& w_3<\cdots<w_p\le w_k<w_2<w_4<\cdots<w_{p-1}\le w_{k-1} \hspace{3mm} \text{ if $p$ is odd} \\
w_1 &<& w_3<\cdots<w_{p-1}\le w_k<w_2<w_4<\cdots<w_{p}\le w_{k-1} \hspace{3mm} \text{ if $p$ is even.} 
\end{eqnarray*}
\end{claim}
\begin{proof}
By Proposition \ref{prop:basic}(3) and the observation that $\rho(w_i)\ne \lam(w_i)$ for all $i$,  \[w_2=\lam(w_1)<w_1<\rho(w_1)=w_k<w_2 \hspace{3mm} \text{and} \hspace{3mm} w_{1}=\lam(w_k)<w_k<\rho(w_k)=w_{k-1}<w_{1}.\]
		
In total we have $w_k<w_2,w_{k-1}<w_1$.  If $w_k<w_{k-1}<w_2<w_1$, then $w_{k-1},w_2$ would both be in $[w_k,w_2]=[\rho(w_1),\lam(w_1)]$ and $[w_{k-1},w_1]=[\rho(w_k),\lam(w_k)]$. This contradicts Proposition \ref{prop:relativerholam}(2), so we must have $w_{k}<w_{2}\le w_{k-1}<w_{1}$, which establishes the result when $p=2$.

Assume the result holds up to $p\le k$.  For simplicity we only consider the case that $p$ is even, the odd case being completely analogous.   By the inductive hypothesis it suffices to prove $w_{p-2}<w_p\le w_{k-1}$. By Proposition~\ref{prop:basic}(3) and our inductive hypothesis, $w_{p-2}<w_{k-1},w_{p}<w_{p-1}.$

If $w_{p-2}<w_{k-1}<w_{p}<w_{p-1}$, then $w_{k-1},w_{p} \in [w_{p-2},w_{p}]=[\rho(w_{p-1}),\lam(w_{p-1})]$ and $w_{k-1},w_{p} \in [w_{k-1},w_1]=[\rho(w_k),\lam(w_k)]$, a contradiction to Proposition~\ref{prop:relativerholam}(2) as $w_{k-1} \neq w_p$. 
\end{proof}
Applying this claim with $p=k$ gives the result.
\end{proof}

\begin{proof}[Proof of Theorem~\ref{thm:structure}]
The backwards direction follows from Lemma~\ref{lem:backwards}.  For the forwards direction, let $C=(w_1,\ldots,w_{2\ell+1})$ be as in Lemma~\ref{lem:W}.  Note that Lemma~\ref{lem:W} and Proposition \ref{prop:basic}(2) imply that the vertices have the correct relative ordering and that $|[w_i,w_{i-1}]| \geq r$.  Thus for all $i$ there exists an $r$-set $e_i\sub [w_{i},w_{i-1}]$ with $w_i,w_{i-1}\in e_i$, and by Proposition~\ref{prop:basic}(3) and (4), $e_i\in H$ for all $i$.  Thus each edge $h \in H$ must intersect $[w_i,w_{i-1}]$ for all $i \in [2\ell +1]$ in order to not form an $M_1^{(r)}$ with any $e_i$. That is, $ H \subseteq H_n^{(r)}(C)$, and the result follows by using that any $e \in \binom{\Omega_n}{r} \setminus H_n(C)$ necessarily forms an $M_1^{(r)}$ with an edge from $H$ and the assumption that $H$ is a maximal $M_1^{(r)}$-free cgh. 
\end{proof}

\subsection{Optimizing Theorem~\ref{thm:structure}}\label{ss:opt}
We say that a tuple  $C=(w_1,\ldots,w_{2\ell+1})$ of distinct vertices from $\Om_n$ is \textit{semi-valid} if \[w_1<w_3<\cdots<w_{2\ell+1}<w_2<\cdots w_{2\ell}<w_1,\] and that $C$ is \textit{$r$-valid} if moreover $|[w_i,w_{i-1}| \geq r$ for all $i \in [2\ell+1]$.   By Theorem~\ref{thm:structure}, every $M_1^{(r)}$-saturated $r$-graph is of the form $H_n^{(r)}(C)$ such that $C$ is $r$-valid.  Thus to bound $\sat(n,M_1^{(r)})$, it suffices to find an $r$-valid tuple $C$ such that $H_n^{(r)}(C)$ has the fewest number of edges.  To do this, we will start with an arbitrary $r$-valid $C$ and perform certain ``local moves'' which decreases $|H_n^{(r)}(C)|$. 

Our primary local move consists of rotating $k$ consecutive points of $C$ one unit clockwise or counterclockwise.  To this end, we say that a semi-valid tuple $C$ is \textit{$(i,k)$-consecutive} if there exists some $j$ such that $w_{i+2s}=v_{j+s}$ for all $0\le s< k$.  That is, starting from $w_i$ there are $k$ consecutive points from $\Om_n$ in $(w_1,\ldots,w_{2\ell+1})$.  If $C$ is $(i,k)$-consecutive, we define $C_{i,k,m}$ to be the tuple\footnote{not necessarily semi-valid} $(w'_1,\ldots,w'_{2\ell+1})$ with  $w'_{i+2s}=v_{j+s+m}$ for $0\le s<k$ and $w'_{p}=w_p$ for all other $p$.  That is, we shift the $k$ consecutive points $m$ units clockwise (or $-m$ units counterclockwise if $m$ is negative). Our main technical result says that if $C_{i,k,m}$ is semi-valid for $m=\pm 1$, then ``local moves'' decrease the number of edges.

\begin{prop}\label{prop:moveTech}
Let $C=(w_1,\ldots,w_{2\ell+1})$ and $i,k$ be such that $C$ is $(i,k)$-consecutive with $1\le k \le \ell$.  If each of the tuples $C_{i,k,-1},C_{i,k,0},C_{i,k,1}$ are semi-valid, then \[|H_n^{(r)}(C)|\ge \min\{|H_n^{(r)}(C_{i,k,-1})|,|H_n^{(r)}(C_{i,k,1})|\}.\]
Moreover, this inequality is strict if $r\ge 3$ and $n\ge 3r-5$.
\end{prop}
The assumption that $r\ge3$ is necessary for the bound to be strict (since Lemma~\ref{lem:r2} shows that $|H_n^{(2)}(C)|=n$ for all $C$), but it is likely that the dependency on $n$ can be improved upon.  
\begin{proof}
We can assume $i=1$ due to symmetry and let $v_j=w_1$ and $C_{1,k,\pm1}= (w^\pm_1,\ldots,w^{\pm}_{2\ell+1})$. For ease of notation, let $H_m=H_n^{(r)}(C_{i,k,m})$ for $m=-1,0,1$.  Define $S_p=[w_p,w_{p-1}]$ and $T_p^\pm=[w^\pm_p,w^\pm_{p-1}]$ for all $p$.  By definition, $e \in H_0$ if and only if $e$ intersects each interval $S_p$, and similarly $e \in H_{\pm 1}$ if and only if $e$ intersects each interval $T_p^{\pm}$.  We first rewrite these intervals as follows: 
\begin{claim}\label{cl:SpTp}
\[S_p=\begin{cases}
[w_p,w_{p-1}] & p\ne s+1\tr{ with }0\le s\le 2k+1,\\ 
[v_{j+s},w_{2s}] & p=2s+1\tr{ with }0\le s< k,\\ 
[w_{2s+2},v_{j+s}] & p=2s+2\tr{ with }0\le s< k.
\end{cases}\]
Also,
\[T_p^{\pm}=\begin{cases}
[w_p,w_{p-1}] & p\ne s+1\tr{ with }0\le s\le 2k+1,\\ 
[v_{j+s\pm1},w_{2s}] & p=2s+1\tr{ with }0\le s< k,\\ 
[w_{2s+2},v_{j+s\pm1}] & p=2s+2\tr{ with }0\le s< k.
\end{cases}\]
\end{claim}
\begin{proof}
The case of $S_p$ is immediate from the definitions.  The $T_p^{\pm}$ case is almost immediate, but one has to be a little careful and show $w_{2s}\notin \{w_1,w_{3},\ldots,w_{2k-1}\}$ for all $0\le s< k$.  Because $s< k\le \ell$,  either $w_{2s}$ is a vertex of even index or $w_{2s}=w_{2\ell+1}$ if $s=0$.  On the other hand, $ \{w_1,w_{3},\ldots,w_{2k-1}\}$ consists of vertices of odd index with size at most $2\ell-1$ since $k\le \ell$, so the result follows.
\end{proof}
To conclude the result, we will need to compute $|H_m|-|H_{m'}|$ for various values of $m,m'$.  To this end, we define $E_{m,m'}=H_m\sm H_{m'}$ for $m,m' \in \{-1,0,1\}$.
\begin{claim}
Let $e \in \binom{\Omega_n}{r}$ and $m\in \{-1,0\}$.
\begin{itemize}
\item We have $e\in E_{m,m+1}$ if and only if there exists an $s$ with $0\le s<k$ with $e\sub (w_{2s},v_{j+s+m}]$, $v_{j+s+m}\in e$, and $e$ intersects $(w_{2s},w_{2s+2}]$.
\item We have $e\in E_{m+1,m}$ if and only if there exists an $s$ with $0\le s<k$ with $e\sub [v_{j+s+m+1},w_{2s+2})$, $v_{j+s+m+1}\in e$, and $e$ intersects $[w_{2s},w_{2s+2})$.
\end{itemize}
\end{claim}
\begin{proof}
We only prove the result for $E_{0,1}$, the proofs of the other cases being completely analogous (and in fact, the other cases follow from this case after reversing the order of $\Om_n$ and/or setting $C=C_{i,k,\pm1}$). We first show that these conditions are necessary.  
		
Let $e\in E_{0,1}$.  Because $e\in H_0$, it intersects every $S_p$, and because $e\notin H_1$, it is disjoint from $T_p^+$ for some $p$. Note that $S_p\sub T_p^+$ implies $e\cap T_p^+\ne \emptyset$, so by Claim~\ref{cl:SpTp}, $T_p^+ = [v_{j+s+1},w_{2s}]$ for some $0\le s<k$. Because $e$ intersects $S_{2s+1}=[v_{j+s},w_{2s}]$, we have $v_{j+s} \in e$ and $e \cap [v_{j+s+1},w_{2s}] = \emptyset$, i.e.  $e\sub (w_{2s},v_{j+s}]$.  Since $e\in H_0$, it must intersect $[w_{2s+3},w_{2s+2}]$.  Because $s<k\le \ell$,  \[w_{2s+1}=v_{j+s}<v_{j+s+1}\le w_{2s+3}<\cdots<w_{2\ell+1}<w_2<\cdots<w_{2s}<w_{2s+2},\] so the only way $e$ can intersect $[w_{2s+3},w_{2s+2}]$ and not $[v_{j+s+1},w_{2s}]$ is if it intersects $(w_{2s},w_{2s+2}]$.  Thus the conditions in the claim are necessary.
		
We now show that the conditions are sufficient.  Let $e$ be as in the hypothesis. As $e\cap[v_{j+s+1},w_{2s}]=\emptyset$, $e\notin H_1$ by Claim~\ref{cl:SpTp}. Thus it suffices to show that $e\in H_0$, i.e. that it intersects each interval $S_p=[w_p,w_{p-1}]$. We claim that $e$ intersects $S_p$ if $p\in \{2s+1,2s-1,\ldots,2s-2\ell+1\}$.  Indeed, by the relative ordering we have $w_p\le w_{2s+1}\le w_{p-1} <w_p$ for every $p$ in this range, and because $w_{2s+1}=v_{j+s}\in e$, $e\cap [w_p,w_{p-1}]\ne \emptyset$ for these values of $p$. Similarly for $p\in \{2s,2s-2,\ldots, 2s-2\ell+2\}$ we have $[w_{2s},w_{2s+2}]\sub [w_p,w_{p-1}]$, so $e$ intersects these $S_p$. This proves the result.  
\end{proof}
	
In order to compute $|E_{m,m'}|$, for $0\le s<k$ we define
\begin{align*}
t_{s}=|(w_{2s},w_{2s+2})|,\   t_{k}=|(w_{2k},&v_{j})|,\ t_{-1}=|(v_{j+k+1},w_{0})|,\\  t_{> s}=\sum_{s'=s+1}^k t_{s'},\hspace{2em} &t_{<s}=\sum_{s'=-1}^{s-1}t_{s'}. 
\end{align*}
	
For $m\in \{-1,0\}$, let $E^s_{m,m+1}$ be the set of $r$-sets $e$ such that $e\sub (w_{i+2s-1},v_{j+s+m}]$, $v_{j+s+m}\in e$, and $e$ intersects $(w_{2s},w_{2s+2}]$.  By the previous claim, $\bigcup E^s_{m,m+1}= E_{m,m+1}$, and in fact this union is disjoint since $v_{j+s+m}\in e$ implies $e\not\subseteq(w_{2s'-1},v_{j+s'+m}]$ for any $s'<s$, so 
\begin{equation}\sum_s |E^s_{m,m+1}|=|E_{m,m+1}|.\label{eq:EsSum}\end{equation}
	
We claim that
\begin{equation}|E^s_{m,m+1}|=\sum_{p=1}^{r-1} {t_s+1\choose p}{t_{>s}+k+m\choose r-1-p}.\label{eq:mm+1}\end{equation}
Indeed, any $e\in E^s_{m,m+1}$ contains $v_{j+s+m}$, some $1\le p\le r-1$ elements from $(w_{2s},w_{2s+2})\cup \{w_{2s+2}\}$, with the remaining elements coming from \[(w_{2s+2},v_{j+s+m})=(w_{2s+2},w_{2s+4})\cup \cdots \cup (w_{2k},v_{j+m})\cup [v_{j+m},v_{j+s+m}) \cup \{w_{2s+4},w_{2s+6},\ldots,w_{2k+2}\},\]
where we note that $|(w_{2k},v_{j+s+m})|=t_{k}+m,\ |[v_{j+m},v_{j+s+m})|=s$, and $|\{w_{2s+4},w_{2s+6},\ldots,w_{2k+2}\}|=k-s$; proving \eqref{eq:mm+1}. One can define $E^s_{m+1,m}\sub E_{m+1,m}$ in an analogous way and prove
\begin{equation}|E^s_{m+1,m}|=\sum_{p=1}^{r-1} {t_{s}+1\choose p}{t_{<s}+k-m\choose r-1-p},\label{eq:m+1m}\end{equation}
with the only significant difference in the proof being that 
{\small\[(v_{j+s+m+1},w_{2s})=(v_{j+s+m+1},v_{j+k+m+1}]\cup (v_{j+k+m+1},w_{0})\cup (w_{0},w_2)\cup \cdots \cup (w_{2s-2},w_{2s})\cup \{w_0,w_2,\ldots,w_{2s-2}\}.\]}
To prove the first part of the proposition, it suffices to show that
\begin{equation}\label{eq:nonnegativeclaim310}
0 \leq (|H_0|-|H_{-1}|)+(|H_0|-|H_1| = |E_{0,-1}|+|E_{0,1}|-|E_{-1,0}|-|E_{1,0}|.    
\end{equation} 
We now note that using \eqref{eq:EsSum} and then \eqref{eq:mm+1} and \eqref{eq:m+1m}, we get

\begin{align}&|E_{0,-1}|+|E_{0,1}|-|E_{-1,0}|-|E_{1,0}| =\sum_s|E_{0,-1}^s|+|E_{0,1}^s|-|E_{-1,0}^s|-|E_{1,0}^s|\nonumber\\&=\sum_s\sum_{p=1}^{r-1}{t_s+1\choose p}\l({t_{<s}+k+1\choose r-1-p}+{t_{>s}+k\choose r-1-p}-{t_{>s}+k-1\choose r-1-p}-{t_{<s}+k\choose r-1-p}\r).\label{eq:mess} \end{align}	
Observe that each term of this innermost sum is non-negative, so \eqref{eq:nonnegativeclaim310} does indeed hold.  To show that \eqref{eq:nonnegativeclaim310} is strict when $r\ge 3$ and $n\ge 3r-6$, we will find a positive term in the sum. 

We claim that for any $s$ with $t_s\ge r-3$, the $p=r-2\ge 1$ term of \eqref{eq:mess} will be positive.  Indeed, we have for example ${t_{>s}+k\choose 1}-{t_{>s}+k-1\choose 1}=1$, where implicitly here we used $t_{>s}+k\ge 1$ since $k\ge 1$, and we also have ${t_s+1\choose r-2}\ge 1$ since $t_s\ge r-3$.  Thus we can assume $t_s\le r-4$ for all $s$ and will show \begin{equation}n-3=t_s+(t_{>s}+k)+(t_{<s}+k+1).\label{eq:n}\end{equation}  Indeed, by our previous reasoning in computing $|E_{0,1}^s|$ and $|E_{1,0}^s|$, we implicitly showed
\[|(w_{2s+2},v_{j+s})|+|(v_{j+s+1},w_{2s})|=t_{>s}+k+t_{<s}+k,\]
and from this the claim follows after observing \[\Om_n=(w_{2s+2},v_{j+s})\cup (v_{j+s+1},w_{2s})\cup (w_{2s},w_{s+2})\cup \{w_{2s},w_{2s+2},v_{j+s},v_{j+s+1}\}.\]  Since $t_s\le r-4$ for (every) $s$, by \eqref{eq:n} we must have, say, $t_{>s}+k\ge \half (n-r+1)$.  In this case the $p=1$ term for $s$ in \eqref{eq:mess} will be positive provided $n\ge 3r-5$. This implies ${t_{<s}+k\choose r-2}>0$ as desired.
\end{proof}

We now find a set of $k\le \ell$ consecutive points such that Proposition~\ref{prop:moveTech} applies.

\begin{lemma}\label{lem:consecutive}
Let $C$ be a semi-valid sequence of length $2\ell+1$.  If $C$ is not $(i,2\ell+1)$-consecutive for any $i$, then there exists $i,k$ such that $C$ is $(i,k)$-consecutive with $1\le k\le \ell$ and such that $C_{i,k\pm1}$ are both semi-valid.
\end{lemma}
\begin{proof}
Let $k'$ be the largest integer such that there exists $i'$ with $C$ being $(i',k')$-consecutive.  Let $i=i'+2k'$, let $j$ be such that $w_i=v_j$, and let $k$ be the largest integer such that $w_{i+2k-2}=v_{j+k-1}$.  That is, $w_i,w_{i+2},\ldots,w_{i+2k-2}$ consists of a maximal set of consecutive points that appear directly after the $k'$ consecutive points starting with $w_{i'}$.  We claim that $i,k$ satisfy the properties of the lemma.  
	
Indeed, $C$ is $(i,k)$-consecutive by construction.  Also, by construction, $w_{i+2k}\ne v_{j+k}$ and $w_{i-2}\ne v_{j-1}$ (the former is due to choosing $k$ largest, and if the latter were false then $C$ would be $(i',k'+1)$-consecutive, which contradicts our choice of $k'$).  Thus $C_{i,k,\pm 1}$ consists of $2\ell+1$ distinct points. Moreover, the relative order of $C_{i,k,\pm 1}$ is the same as in $C$, so these sequences are semi-valid.
	
By construction $\{w_{i'},w_{i'+2},\ldots,w_{i'+2k'-2},w_i,w_{i+2},\ldots,w_{i+2k-2}\}$ consists of $k'+k$ distinct vertices.  Thus $k'+k\le 2\ell+1$, and $k\le k'$ by definition of $k'$, so $k\le \ell$ as desired.
\end{proof}
Whenever $n$ is understood, we define $\semi_{2\ell+1}$ to be the set of semi-valid tuples from $\Omega_n$ of length $2\ell+1$ and $\valid_{2\ell+1}$ to be the set of $r$-valid tuples from $\Omega_n$ of length $2\ell+1$.  Lemma~\ref{lem:consecutive} and Proposition~\ref{prop:moveTech} immediately gives the following.

\begin{cor}\label{cor:compress}
If $r\ge 3$ and $n\ge 3r-4$, then a tuple $C\in \semi_{2\ell+1}$ satisfies \[|H_n^{(r)}(C)|=\min_{C'\in \semi_{2\ell+1}} |H_n^{(r)}(C')|\] if and only if there exists an $i$ such that $C$ is $(i,2\ell+1)$-consecutive.
\end{cor}
Corollary~\ref{cor:compress} characterizes the optimal semi-valid sequence for any given $\ell$.  The next proposition shows that $\ell$ is optimized when it is chosen as small as possible.

\begin{prop}\label{prop:length}
If $\ell\ge 1$, $r\ge 3$, and $n\ge \max\{3r-5,2\ell+3\}$, then
\[\min_{W\in \semi_{2\ell+1}} |H_n^{(r)}(C)|<\min_{W\in \semi_{2\ell+3}} |H_n^{(r)}(C)|.\]
\end{prop}
We note that the lower bound $n\ge 2\ell+3$  ensures that $\semi_{2\ell+3}$ is a non-empty set.
\begin{proof}
For $m \in \{ \ell,\ell+1\}$, let \[C_m=(v_{n-m},v_{1},v_{n-m+1},v_2,\ldots,v_{n-1},v_{m},v_0),\hspace{2em} H_m=H_n^{(r)}(C_m).\]  Observe that $C_m$ is $(1,2m+1)$-consecutive, so  $|H_n^{(r)}(C_m)|=\min_{W\in \semi_{2m+1}} |H_n^{(r)}(C)|$ by Corollary~\ref{cor:compress}.  Thus it suffices to prove that $|H_\ell|<|H_{\ell+1}|$. Let \[S_p=\begin{cases}[v_{n-\ell+(p-1)/2},v_{(p-1)/2}] & 1\le p\le 2\ell+1\tr{ odd,}\\
[v_{p/2},v_{n-\ell-1+p/2}]  & 2\le p\le 2\ell\tr{ even.}\end{cases}\] 
\[T_p=\begin{cases}[v_{n-\ell-1+(p-1)/2},v_{(p-1)/2}] & 1\le p\le 2\ell+3\tr{ odd,}\\
[v_{p/2},v_{n-\ell-2+p/2}]  & 2\le p\le 2\ell+2\tr{ even.}\end{cases}\] 
With this an $r$-set $e$ is an edge of $H_\ell$ if and only if it intersects each $S_p$ set, and it is an edge of $H_{\ell+1}$ if and only if it intersects each $T_p$ set.

\begin{claim}\label{cl:Ek}
For $0\le k\le \ell$, let $E_k$ denote the set of $r$-sets $e$ which have $e\sub  [v_{k+1},v_{n-\ell+k-1}]$ and $v_{k+1},v_{n-\ell+k-1}\in e$.  Then $\bigcup E_k\sub H_{\ell+1}\sm H_\ell$.
\end{claim}
\begin{proof}
Observe that any $e\in E_k$ has $e\cap S_p=\emptyset$ for $p=2k+1$ (since in this case $1\le p\le 2\ell+1$).  Thus $E_k\cap H_\ell=\emptyset$, and it suffices to prove that each $e\in E_k$ intersects each interval $T_p$.
	
Observe that $v_{k+1}\in T_p$ for all $p\le 2k+2$ even and $p\ge 2k+3$ odd; and $v_{n-\ell+k-1}\in T_p$ for all $p\ge 2k+2$ even and $p\le 2k+1$ odd (and in all these cases, $1\le p\le 2\ell+3$).  This means every $e\in E_k$ intersects every interval $T_p$, proving the result.
\end{proof}
	
It is not difficult to see that the $E_k$ sets are disjoint and that $|E_k|={n-\ell-3\choose r-2}$, so 
\begin{equation}|H_{\ell+1}\sm H_\ell|\ge (\ell+1){n-\ell-3\choose r-2}.\label{eq:ell+1ell}\end{equation}

\begin{claim}\label{cl:Ek'}
For $0\le k\le \ell$, let $E'_k$ consist of the $r$-sets $e$ with $e\sub [v_{n-\ell+k},v_{k}]$ and $v_{n-\ell+k},v_{k}\in e$.  Then $H_\ell \sm H_{\ell+1}\sub \bigcup_{k=0}^\ell E_k'$.
\end{claim}
\begin{proof}
Let $e\in H_{\ell}\sm H_{\ell+1}$, which in particular implies there is some $q$ such that $e\cap T_q=\emptyset$.  Observe that $T_p\supset S_p$ for all $p$ odd.  Because $e$ intersects $S_q$, we must have that $q$ is even, so $e$ disjoint from $T_q$ means $e\sub [v_{n-\ell-1+q/2},v_{q/2-1}]$.  Because $e$ intersects $S_q$ and $S_{q-2}$, we must have $v_{n-\ell-1+q/2},v_{q/2-1}\in e$, so in total $e\in E'_{q/2-1}$.  This proves $H_{\ell}\sm H_{\ell+1}\sub \bigcup E'_k$
\end{proof}
Because the $E'_k$ sets are disjoint, we find
\[|H_{\ell}\sm H_{\ell+1}|\le(\ell+1){\ell-1\choose r-2}.\]
Using this \eqref{eq:ell+1ell}, and $r\ge 3$, we see that $|H_\ell|<|H_{\ell+1}|$ provided \[\ell-1<n-\ell-3\iff n>2\ell+2,\]
and this holds by our choice of $n$.
\end{proof}
Proposition~\ref{prop:length} will be enough to prove our asymptotic result for general $r$.  To prove the result for $r=3$, we need to consider $r$-valid sequences.  In particular, we make the following observation.

\begin{lemma}\label{lem:valid}
If $\ell \ge r-1$, then $\semi_{2\ell+1}=\valid_{2\ell+1}$.
\end{lemma}
\begin{proof}
Let $C=(w_1,\ldots,w_{2\ell+1})\in \semi_{2\ell+1}$.  We wish to prove that $|[w_i,w_{i-1}]|\ge r$ for all $i$, and by the symmetry of the problem it suffices to show $|[w_1,w_{2\ell+1}]| \geq r$.  And indeed, this interval contains all of the $\ell+1\ge r$ elements $w_i$ with $i$ an odd index.
\end{proof}
When $r=3$ and $\ell=1$ we need a slightly different argument.
\begin{lemma}\label{lem:ell1}
Let $C=(v_0,v_{n-2},v_2)$.  For $n\ge 6$, we have
\[|H_n^{(r)}(C)|=\min_{C'\in \vaLID_{3}}|H_n^{(r)}(C')|.\]
\end{lemma}
Here the $n\ge 6$ condition is needed to ensure that $\vaLID_{3}$ is non-empty.
\begin{proof}
Let $C'=(w_1,w_2,w_3)\in  \vaLID_{3}$ be any sequence achieving the minimum.  Note that there is at least one point in between each of $w_i,w_{i-1}$ because the sequence is $3$-valid.  If $w_1$ has exactly one vertex between $w_2$ and $w_3$, then $C'$ is a rotation of $C$, so we can assume this is not the case.  If $w_1$ has more than one vertex in between both $w_2$ and $w_3$, then $C'_{i,1,\pm1}$ is $3$-valid, so by Proposition~\ref{prop:moveTech}, $C'$ does not achieve the minimum, a contradiction.  Thus we can assume that, say, $w_1=v_0,\ w_2=v_{n-2}$, and $w_3\ne v_{2}$.  But by a symmetric argument $w_3$ can not have more than one vertex in between both $w_1$ and $w_2$, so $w_3=v_{n-4}$, which is again a rotation of $C$, proving the result.
\end{proof}

We now prove our main result.
\begin{proof}[Proof of Theorem~\ref{thm:mainsat}]
The $r=2$ case follows from Lemma~\ref{lem:r2} and Theorem~\ref{thm:structure}, so we may that assume $r\ge 3$. The upper bound follows from considering $H_n^{(r)}$ from \eqref{eq:const}.  Indeed, it is not difficult to check $|H_n^{(r)}|\sim {n \choose r-1}$ and $|H_n^{(3)}|={n-1\choose 2}+3n-11$.  Moreover, $H_n^{(r)}=H_n^{(r)}(C)$ for $C=(w_1,w_2,\ldots,w_{2r-1})=(v_{n-r+1},v_{1},v_{n-r+2},v_2,\ldots,v_{n-1},v_{r-1},v_0)$, so by Theorem~\ref{thm:structure} these cgh's are $M_1^{(r)}$-saturated. 
	
For the lower bound, by Theorem~\ref{thm:structure}, 
\begin{equation}\label{eq:maximalopt}
\csat(n,M_1^{(r)})=\min_\ell \min_{C\in \valid_{2\ell+1}}|H_n^{(r)}(C)|\ge \min_\ell \min_{C\in \semi_{2\ell+1}}|H_n^{(r)}(C)|.    
\end{equation}

If $n\ge 3r-5$, then by Proposition~\ref{prop:length} and Corollary~\ref{cor:compress}, the right hand side of \eqref{eq:maximalopt} is minimized by $C=(v_0,v_{n-1},v_1)$, and in particular $|H_n^{(r)}(C)|\ge {n-1\choose r-1}$ since $H_n^{(r)}(C)$ contains every $r$-set containing $v_0$.  This proves the asymptotic lower bound for all $r$.
	
To prove the equality for $r=3$, we have by Lemmas~\ref{lem:valid} and \ref{lem:ell1}, Corollary~\ref{cor:compress}, and Proposition~\ref{prop:length} that the minimum will be achieved by either $C_1=(v_0,v_{n-2},v_2)$ or $C_2=(v_0,v_{n-2},v_1,v_{n-1},v_2)$, and further that if $C_1$ does not achieve the minimum, then up to rotation $H_n^{(3)}(C_2)$ will be the unique minimizer.  Let $H_m=H_n^{(r)}(C_m)$ with $m=1,2$.  Using the same reasoning as in Claims~\ref{cl:Ek} and \ref{cl:Ek'}, one can show the following.
\begin{claim}
We have that $H_1\sm H_2$ contains every triple $e\sub [v_2,v_{n-2}]$ with $v_2,v_{n-2}\in e$, and $H_2\sm H_1= \{ \{v_{n-1},v_0,v_1\} \} $. 
\end{claim}
This claim implies that $|H_1|-|H_2|\ge (n-5)-1$, proving that $H_2$ is a minimizer for $n\ge 6$ and the unique minimizer when $n>6$.  
\end{proof}

\section{Concluding Remarks}
In Theorem~\ref{thm:mainsat} we determined $\csat(n,M_1^{(r)})$ asymptotically for all $r$, and moreover we showed that for $r=3$, the cgh $H_n^{(3)}$ defined in \eqref{eq:const} is the unique extremal construction for $n$ sufficiently large.  We believe that $H_n^{(r)}$ from \eqref{eq:const} minimizes $|H_n^{(r)}(C)|$ amongst $r$-valid tuples $C$ for larger $r$:
\begin{conj}
For all $r\ge 3$ and $n$ sufficiently large in terms of $r$, we have
\[\csat(n,M_1^{(r)})=|H_n^{(r)}|,\]
where $H_n^{(r)}$ is as defined in \eqref{eq:const}.  Moreover, $H_n^{(r)}$ is the unique $M_1^{(r)}$-saturated cgh achieving this bound provided $n$ is sufficiently large in terms of $r$.
\end{conj}

In Theorem~\ref{thm:order} we determined the order of magnitude of $\csat(n,F)$ for every two edge $3$-cgh, except for the case $F=S_3$.  We believe that our upper bound is closer to the truth.
\begin{conj}\label{conj:S3}
\[\csat(n,S_3)=\Theta(n\log_2 n).\]
\end{conj}
We note that it is somewhat rare for saturation type problems to exhibit growth rates involving logarithmic terms.  The only such example we are aware of comes from a rainbow saturation problem introduced by Barrus et. al.~\cite{BFVW},  which was resolved independently by Ferrara et. al.~\cite{Ferrara},  Gir$\tilde{\mathrm{a}}$o, Lewis and Popielarz~\cite{GLP}, and Kor\'andi~\cite{Korandi}.  All of these papers use different techniques, and it is plausible that similar ideas could be useful in resolving Conjecture~\ref{conj:S3}.

\begin{figure}
\begin{center}
\begin{tabular}[c]{|c|c|c|c|c|c|}
\hline
& & & & &    \\
$F$ & Diagram &  $\cex(n,F)$ \cite{FMOV} & $\lesssim \csat(n,F) $  & $\csat(n,F) \lesssim $  & {\small Best Guess $\csat(n,F)$} \\ & & & &  & \\ \hline
\hline
\raisebox{0.3in}{$M_1$} & \includegraphics[width= 20mm, height= 20mm, trim = {0 -0.1cm 0 -0.1cm}]{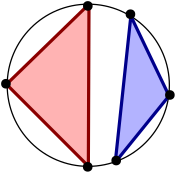}
& \raisebox{0.3in}{$\sim \frac{n^3}{24}$} 
& \raisebox{0.3in}{$\binom{n}{2}$} & \raisebox{0.3in}{$\binom{n}{2}$}  & \raisebox{0.3in}{--} \\
\hline 

\raisebox{0.3in}{$M_2$} & \includegraphics[width= 20mm, height=20mm,trim = {0 -0.1cm 0 -0.1cm}]{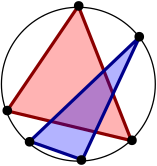}
& \raisebox{0.3in}{$\binom{n}{2}-2$} & \raisebox{0.3in}{$\frac{11n}{9}$} & \raisebox{0.3in}{$3n$}  & \raisebox{0.3in}{$3n$} \\

\hline
\raisebox{0.3in}{$M_3$} & \includegraphics[width= 20mm, height=20mm,trim = {0 -0.1cm 0 -0.1cm}]{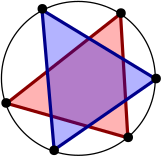}
& \raisebox{0.3in}{$\binom{n}{3} - \binom{n-3}{3}$} & \raisebox{0.3in}{$n^2$} & \raisebox{0.3in}{$\binom{n}{3} - \binom{n-3}{3}$} & \raisebox{0.3in}{${n\choose 3}-{n-3\choose 3}$} \\ \hline 

\raisebox{0.3in}{$S_1$} & \includegraphics[width= 20mm, height=20mm,trim = {0 -0.1cm 0 -0.1cm}]{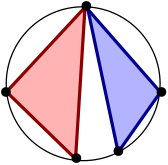}
& \raisebox{0.3in}{$\sim \frac{n^3}{24}$} & \raisebox{0.3in}{$\f{n}{2}$} & \raisebox{0.3in}{$n$}  & \raisebox{0.3in}{$n$} \\
\hline

\raisebox{0.3in}{$S_2$} & \includegraphics[width= 20mm, height= 20mm,trim = {0 -0.1cm 0 -0.1cm} ]{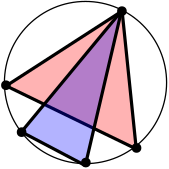}
& \raisebox{0.3in}{$[\lfloor \frac{n^2}{4} \rfloor -1 ,\frac{23n^2}{64}]$} & \raisebox{0.3in}{$\frac{2n}{5}$} & \raisebox{0.3in}{$n$} & \raisebox{0.3in}{$n$} \\ 
\hline

\raisebox{0.3in}{$S_3$} & \includegraphics[width= 20mm, height= 20mm,trim = {0 -0.1cm 0 -0.1cm} ]{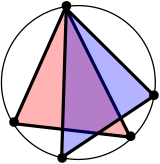}
& \raisebox{0.3in}{$\displaystyle{{\frac{n(n-2)}{2}}\atop{\mbox{for }n\mbox{ even}}}$} & \raisebox{0.3in}{$\frac{3n}{2}$} & \raisebox{0.3in}{$3n\log_2(n)$}  & \raisebox{0.3in}{$\Theta(n\log_2 n)$} \\ 
\hline 

\raisebox{0.3in}{$D_1$} & \includegraphics[width= 20mm, height= 20mm,trim = {0 -0.1cm 0 -0.1cm} ]{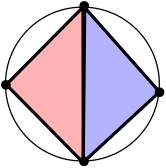}
& \raisebox{0.3in}{$\sim \frac{n^3}{24}$} & \raisebox{0.3in}{$\frac{n^2}{6}$} & \raisebox{0.3in}{$\frac{n^2}{4}$} & \raisebox{0.3in}{???} \\ 
\hline 

\raisebox{0.3in}{$D_2$} & \includegraphics[width= 20mm, height= 20mm,trim = {0 -0.1cm 0 -0.1cm} ]{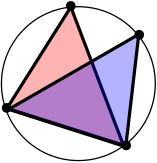}
& \raisebox{0.3in}{
                     $\displaystyle{\frac{2n^2-3n}{9}
                   }\atop{\mbox{for }n \equiv 6 \mod 9} $} & \raisebox{0.3in}{$\frac{n^2}{12}$} & \raisebox{0.3in}{$\frac{5n^2}{24}$} & \raisebox{0.3in}{???} \\ 
\hline
\hline
\end{tabular}
\end{center}
\caption{Table outlining for each $F$ the value of $\cex(n,F)$ given by \cite{FMOV}, our best asymptotic bounds for $\csat(n,F)$, and our best guess as to the true asymptotic value.}   \label{tab:summary}
\end{figure}

For our proof of Theorem~\ref{thm:order}, we tried to give as simple an argument as needed to obtain the correct bounds for the order of magnitudes.  In the Appendix we give arguments which lead to better asymptotic bounds, but unfortunately we still do not have any tight bounds except for $\csat(n,M_1)$.  Table~\ref{tab:summary} summarizes the results we are able to prove, together with our best guesses for the asymptotic behavior of $\csat(n,F)$ in each of these cases.  We highlight a few of these guesses as conjectures.
\begin{conj}
\[\csat(n,S_1)\sim n.\]
\end{conj}
\begin{conj}
\[\csat(n,S_2)\sim n.\]
\end{conj}
We note that there are many constructions which achieve this upper bound for $S_1$ and $S_2$. For $S_1$, if $4|n$, then one can take $H$ to be the disjoint union of any set of $n/4$ copies of $K_4^{(3)}$ such that the convex hulls of their vertex sets are disjoint.  For example, when $n=12$ one can have $K_4^{(3)}$'s on $\{v_0,v_1,v_6,v_7\},\{v_2,v_3,v_4,v_5\},\{v_8,v_{9},v_{10},v_{11}\}$.  

For the matching $M_3$, we suspect that the upper bound is asymptotically tight.

\begin{conj}
\[\csat(n,M_3)\sim \f{3}{2}n^2.\]
\end{conj}
In view of \cite{FMOV}, this conjecture is equivalent to saying $\csat(n,M_3)\sim\cex(n,M_3)$. In fact, it may even be the case that $\csat(n,M_3)=\cex(n,M_3)$, and further that this equality continues to hold for the natural $r$-uniform generalization of $M_3$ which has edges $\{v_0,v_2,\ldots,v_{2r-2}\},\{v_1,v_3,\ldots,v_{2r-1}\}$.  Despite this bold conjecture, we are unable to prove a lower bound for $\csat(n,M_3)$ which is asymptotically better than the trivial lower bound using that there exist $n(n-3)$ edges which can never form an $M_3$, and hence must always be in an $M_3$-saturated cgh.

We suspect that our construction for $M_2$ showing $\csat(n,M_2)\le 3n+O(n)$ is asymptotically tight, but we are not confident enough to pose this as a conjecture.  The situation with $D_1$ and $D_2$ is even less clear.  This is in part due to the following construction which we are unable to analyze: take $H'$ to be a Steiner triple system (i.e. a collection of $3$-sets which contains each pair of vertices exactly once), and then let $H$ be a $D_1$ or $D_2$-saturated cgh containing $H$.  It seems plausible that there might exist a choice of $H'$ so that $|H|=(1+o(1))\f{n^2}{6}$, but we were unable to prove such a result.
\begin{question}
Is $\csat(n,D_i)\sim n^2/6$ for either $i=1,2$?
\end{question}

\textbf{Acknowledgements.}  The authors would like to thank Sean English and Jacques Verstra\"ete for fruitful discussions.  We especially thank Sean English for pointing us to several references, as well as for his proof improving our asymptotic lower bound on $\csat(n,S_1)$.   The second author was funded by the National Science Foundation Graduate Research Fellowship under Grant No. DGE-1650112.

\section*{Appendix}
Below we improve the asymptotic bounds of $\csat(n,F)$ compared to the implicit bounds given in the proof of Theorem~\ref{thm:order}. 
We start with a set theoretic lemma which does not involve any geometry:  

\begin{lemma}\label{lem:stablitymatchin}
If $H \subset \binom{[n]}{3}$ is such that $\cup_{h \in H} h = [n]$ and each $h \in H$ contains a vertex $v\in [n]$ with $d_H(v)\ge 2$, then $|H| \geq 2n/5$. 
\end{lemma}

\begin{proof}
For any $h = \{a,b,c\} \in H$, define $\text{wt}(h) = 1/d_H(a) + 1/d_H(b)+ 1/d_H(c)$. By the given property, every $h \in H$ is so that $\text{wt}(h) \leq 5/2 $ and hence 
\begin{equation}\label{eq:weight}
n = \sum_{v \in \Omega_n} \sum_{v \in h \in H} \frac{1}{d_H(v)} = \sum_{e \in H} \text{wt}(h) \leq |H| \cdot \frac{5}{2}. \qedhere    
\end{equation}
\end{proof}

\begin{prop}\label{prop:S1}
\[2n/5+o(n)\le \csat(n,S_1)\le n+o(n).\]
\end{prop}
\begin{proof}
The upper bound follows from the proof of Theorem~\ref{thm:order}.  For the lower bound, let $H \subset \binom{\Omega_n}{3}$ be an $S_1$-saturated cgh. It is straightforward to see that $|\cup_{h \in H} h|\geq n-2$, as otherwise there exists three isolated vertices and we may add the corresponding triple. If every edge of $H$ contains a vertex of degree two, then the result follows from Lemma \ref{lem:stablitymatchin}, so it suffices to consider the case where there exists $h=\{a,b,c\} \in H$ with $a<b<c<a$ which is disjoint from every other edge in $H$. 

For each vertex $v \in (a,b)$, consider $e_v= \{v,b,c\}$, which is not in  $H$ by hypothesis on $h$. Because $H$ is $S_2$-saturated, $e_v$ forms an $S_2$ with some $h' \in H$. Since every edge in $H \setminus \{h\}$ is disjoint from $h$, it follows that $h'=\{ w,v,x\}$ with $a<w<v<x<b$ or $a<v<b<w<x$. Continuing this argument in $(b,c)$ and $(c,a)$, this implies that $|H| \geq (1/2+o(1))n$.
\end{proof}
A similar argument shows $\csat(n,S_2)\ge 2n/5+o(n)$, but a sharper argument due to Sean English \cite{SEAN} gives the following:
\begin{prop}
\[ n/2 +o(n)\le \csat(n,S_2)\le n+o(n).\]
\end{prop}
\begin{proof}
The upper bound follows from the proof of Theorem~\ref{thm:order}. For the lower bound, let $H \subset \binom{\Omega_n}{3}$ be an $S_2$-saturated cgh. For any $h = \{a,b,c\} \in H$, define $\text{wt}(h) = 1/d_H(a) + 1/d_H(b)+ 1/d_H(c)$. 

\begin{claim}\label{claim:oneheavyedge}
There exists at most one $h \in H$ so that $\text{wt}(h)>2$.
\end{claim}

\begin{proof}
Seeking a contradiction, suppose $h_1 = \{a_1, b_1, c_1\}, h_2 = \{a_2, b_2, c_2\}$ are such that $\text{wt}(h_1)>2$ and $\text{wt}(h_2)>2$. Possibly by relabeling vertices, we assume that $d_H(a_1)=d_H(a_2)=d_H(b_1)=d_H(b_2)=1$.

\textit{Case 1:} $b_2 <b_1 <a_1<a_2<b_2$ \\
Consider $\{a_1,b_1,a_2\} ,\{a_1,b_1,b_2\} $.  These triples are not in $H$ by our degree assumptions, so it follows that these triples must form an $S_2$ with an edge $h\in H$. Again by the degree assumptions we must have $h=\{a_2,b_2,c_2\}$, which implies $a_1<c_2<a_2<a_1$ and $b_2<c_2<b_1<b_2$; a contradiction.

\textit{Case 2:} $b_2 <a_1<a_2<b_1<b_2$ \\
Consider $\{a_1,b_1,a_2\},\{a_1,b_1,b_2\}$.  These triples are are not in $H$ by our degree assumptions, so it follows that these triples form an $S_2$ with an edge $ h \in H$.  Again by the degree assumptions  $h=\{a_2,b_2,c_2\}$, which implies $b_1<c_2<a_1<b_1$ and $a_1<c_2<b_1<a_1$ respectively; a contradiction. \qedhere
\end{proof}
The desired lower bound then follows by using a similar weight argument as in \eqref{eq:weight} together with the fact that there are at most two isolated vertices in any $S_2$-saturated cgh.
\end{proof}

\begin{prop}\label{prop:S3}
\[3n/2+o(n)\le \csat(n,S_3)= O(n\log_2 n).\]
\end{prop}
\begin{proof}
The upper bound follows from Theorem~\ref{thm:order}. For the lower bound, let $H$ be an $n$-vertex $S_3$-saturated cgh.  First observe that no 3-set $\{v_i,v_{i+1},v_{i+2}\}$ can be used to form an $S_3$, so all of these 3-sets are edges in $H$.  We refer to all other edges of $H$ as \textit{non-trivial.}
    
We claim that for all $i$, the vertex $v_i$ is in at least one non-trivial edge which contains either $v_{i-1}$ or $v_{i+1}$.  Indeed, if $\{v_{i-1},v_i,v_{i+2}\}\in H$ then we are done, and otherwise $H$ being $S_3$-saturated implies that this 3-set forms an $S_3$ with some edge $h \in H$. It is not too difficult to see that such an edge $h$ is of the form $\{v_{i},v_{i+1},v_j\}\in H$ with $j\ne i-1,i+2$.  This proves the claim.
    
With this claim, for every $1\le i\le \floor{n/2}$, there is a non-trivial edge containing exactly two consecutive vertices of $\{v_{2i-1},v_{2i},v_{2i+1}\}$, and note that each of these $\floor{n/2}$ edges are distinct from each other (since they are identified uniquely by the even number $j$ such that $v_{j}$ and $v_{j\pm 1}$ are in the edge).  Thus $H$ contains at least $\floor{n/2}$ non-trivial edges, proving the bound.
\end{proof}
We suspect that one can further optimize the argument to give a slightly better asymptotic lower bound, but we omit doing so since we believe $\csat(n,S_3)=\Theta(n\log_2 n)$.
\begin{prop}
\[11n/9+o(n)\le \csat(n,M_2)\le 3n+o(n).\]
\end{prop}
\begin{proof}
The upper bound follows from the proof of Theorem~\ref{thm:order}.  For the lower bound, let $H$ be an $n$-vertex $M_2$-saturated cgh.  First observe that no 3-set $\{v_i,v_{i+1},v_{i+2}\}$ can be used to form an $M_2$, so all of these 3-sets are edges in $H$.  We refer to all other edges of $H$ as \textit{non-trivial.}
    
We claim that for all $i$, at least two of the vertices $\{ v_{i-1},v_i,v_{i+1} \}$ are contained in a non-trivial edge.  Indeed, if $\{v_{i-1},v_{i},v_{i+2}\}\in H$ then we are done, and otherwise $H$ being $M_2$-saturated implies that this 3-set forms an $M_2$ with some edge $ h \in H$. It is not too difficult to see that such an edge $h$ is of the form $\{v_{i+1},v_j,v_k\}\in H$ with $j,k\ne i,i+2,i+3$, so $v_{i+1}$ is in a non-trivial edge.  Similarly if $\{v_{i-2},v_i,v_{i+1}\}\in H$ we are done, and otherwise $v_{i-1}$ is in a non-trivial edge.
    
With this claim, for every $1\le i\le \floor{n/3}$, there exist two vertices of $\{v_{3i-2},v_{3i-1},v_{3i}\}$ such that there exist non-trivial edges $h_i^1,h_i^2$ containing these vertices (possibly with $h_i^1=h_i^2$).  Each non-trivial edge appears at most three times in the list $h_1^1,h_1^2,h_2^1,h_2^2,\ldots$, so in total the number of non-trivial edges is at least $2/3\cdot \floor{n/3}\ge 2n/9-1$, and adding this to the number of trivial edges gives the result.
\end{proof} 
One can likely use a sharper analysis to improve the lower bound for $\csat(n,M_2)$ somewhat, but seemingly without new ideas one can not hope to prove bounds much stronger than, say, $\f{5}{3}n$.  

\begin{prop}
\[n^2/6+O(n)\le \csat(n,D_1)\le n^2/4+O(n).\]
\end{prop}
\begin{proof}
The lower bound follows from the proof of Theorem~\ref{thm:order}. For the upper bound, define the cgh $H$ by including the $3$-sets of the form $\{v_i,v_j,v_{n-i-1}\}$ where $v_i<v_j<v_{n-i-1}$ for all $i\le \floor{n/2}-1$.  It is not difficult to see that the only pairs $\{v_i,v_j\}$ which are in more than one edge in $H$ are those with $j=n-i+1$, and no such pair is in a $D_1$, so $H$ is $D_1$-free.  The number of edges in $H$ is equal to 
\[\sum_{i=0}^{\floor{n/2}-1} (n-2i-2)=n^2/2-2{n/2\choose 2}+O(n)=n^2/4+O(n),\] so it suffices to prove that $H$ is $D_1$-saturated.
    
Let $e=\{v_i,v_j,v_k\} \notin H$, say with $v_i<v_j<v_k<v_i$.  First consider the case $v_i\in [v_0,v_{\floor{n/2}}]$.  If $v_k \in (v_i, v_{n-i-1})$, then we observe that $e$ and $\{v_i,v_k, v_{n-i-1}\} \in H$ form a $D_1$. Thus, we may assume that $v_k \in (v_{n-i-1},v_i)$.  If $v_0 \leq v_k<v_i<v_0$, then $e$ and $\{v_k, v_{n-k-1},v_j\} \in H$ form a $D_1$, and if $v_k<v_0<v_i<v_k$, then $e$ and $\{v_k, v_{n-k+1},v_i\}$ form a $D_1$.  Thus if $v_i\in [v_0,v_{\floor{n/2}}]$ then $e$ cannot be added to $H$ without creating a $D_1$, and an analogous argument gives the same conclusion if $v_i \notin [v_0,v_{\floor{n/2}}]$.  Thus $H$ is $D_1$-saturated, proving the result.
\end{proof}

\begin{prop}\label{prop:D2Upper}
\[n^2/12+O(n)\le \csat(n,D_2)\le 5n^2/24+O(n).\]
\end{prop}
\begin{proof}[Sketch of Proof]
The lower bound follows from the proof of Theorem~\ref{thm:order}.  For the upper bound, let $H'$ be the 3-graph which contains every $\{v_i,v_j,v_{i+j}\}$ with $i\ge 1$ and $2i\le j\le n-i$.  For brevity we omit the proofs of the following claims.
	
\begin{claim}\label{cl:H'}
$H'$ is $D_2$-free and $|H'|=\rec{6}n^2+O(n)$.
\end{claim}
Let $H$ be any $D_2$-saturated cgh which contains $H'$.  We then consider the following pairs of vertices $ \c{P}:=\{ \{v_i,v_j \}:0 \le i<j\le n-1,\ 3i/2 \le j\le 2i \}$.
\begin{claim}\label{cl:edgeto2}
If $\{v_i,v_j,v_k\}\in H\sm H'$ with $v_0\le v_i<v_j<v_k \le v_0$, then $\{v_i,v_j\} , \{v_j,v_k \}\in \c{P}$ and $\{v_i,v_k \}\notin \c{P}$. Moreover, every $\{v_i,v_j\} \in \c{P}$ is in at most edge of $H\sm H'$.
\end{claim}
Let $\c{P}'\sub \c{P}$ be the pairs which lie in an edge of $H\sm H'$. 	By Claim~\ref{cl:edgeto2}, we define a map $f:\c{P}' \to H\sm H'$ such that $f(\{v_i,v_j\})$ is the unique edge of $H\sm H'$ containing the pair $\{ v_i,v_j\} $.  Every $ h\in H\sm H'$ has preimage exactly two under $f$ by Claim~\ref{cl:edgeto2}.  Using this and $|\c{P}|= n^2/12+O(n)$,  \[ |H\sm H'|\le \half |\c{P}'|\le \rec{24}n^2+O(n) . \]  
Combining this with Claim~\ref{cl:H'} shows that $|H|\le 5n^2/24+O(n)$, giving the desired bound.
\end{proof}

A more explicit saturation of the cgh $H'$ used in Proposition~\ref{prop:D2Upper} could yield a better asymptotic upper bound for $\csat(n,D_2)$.  Seemingly any such construction can give at best a bound of $(\rec{6}+\ep)n^2+o(n^2)$ for some $\ep>0$.

\end{document}